\documentclass{amsart}

\usepackage[latin1]{inputenc}
\usepackage[english]{babel}
\usepackage{indentfirst}
\usepackage{amssymb}
\usepackage{amsthm}

\newcommand{\pten}{\ensuremath{\widehat{\otimes}_\pi}}
\newcommand{\N}{\mathbb{N}}
\newcommand{\sub}{\subseteq}
\def\epsilon{\varepsilon}

\newtheorem{theo}{Theorem}[section]
\newtheorem{lem}[theo]{Lemma}
\newtheorem{pro}[theo]{Proposition}
\newtheorem{cor}[theo]{Corollary}
\newtheorem{defi}[theo]{Definition}
\newtheorem{rem}[theo]{Remark}
\newtheorem{exa}[theo]{Example}
\newtheorem{assumption}[theo]{Assumption}

\numberwithin{equation}{section}

\title{On weak compactness in projective tensor products}

\author{Jos\'{e} Rodr\'{i}guez}

\address{Dpto. de Ingenier\'{i}a y Tecnolog\'{i}a de Computadores\\Facultad de Inform\'{a}tica\\
Universidad de Murcia\\ 30100 Espinardo (Murcia)\\ Spain} \email{joserr@um.es}



\thanks{The research is partially supported by grants MTM2017-86182-P 
(funded by MCIN/AEI/10.13039/501100011033 and ``ERDF A way of making Europe'') and 
20797/PI/18 (funded by {\em Fundaci\'on S\'eneca})}

\begin{document}

\begin{abstract}
We study the property of being strongly weakly compactly generated (and some relatives) in projective tensor products of Banach spaces. 
Our main result is as follows. Let $1 < p,q<\infty$ be such that $1/p+1/q\geq 1$. Let $X$ (resp.,~$Y$) be a Banach space with a countable unconditional finite-dimensional Schauder decomposition having a disjoint lower $p$-estimate (resp.,  $q$-estimate). If $X$ and $Y$ are strongly weakly compactly generated, then so is 
its projective tensor product $X \pten Y$.
\end{abstract}

\maketitle

\section{Introduction}

In general, it is not easy to deal with the weak topology 
in projective tensor products of Banach spaces, see~\cite{avi-mar-rod-rue} for recent developments and further background 
on this topic. For instance, given two relatively weakly compact sets $W_1 \sub X$ and $W_2 \sub Y$ of the Banach spaces $X$ and~$Y$, the set
$$
	W_1 \otimes W_2:=\{x\otimes y: \, x\in W_1, \, y\in W_2\}
$$ 
might not be relatively weakly compact in the projective tensor product $X\pten Y$.
Such an example is obtained when $W_1=W_2=\{e_n:n\in \N\}$ is the usual basis of~$X=Y=\ell_2$, because in this case the sequence
$(e_n\otimes e_n)_{n\in \N}$ in $\ell_2 \pten \ell_2$ is equivalent to the usual basis of~$\ell_1$. 

In this note we discuss the property of being {\em strongly weakly compactly generated} ({\em SWCG} for short) in projective tensor products. A
Banach space~$X$ is called SWCG if there is a weakly compact set $G \sub X$ such that, 
for every weakly compact set $K \sub X$ and for every $\varepsilon>0$, there is $n\in \mathbb{N}$ such that $K \subseteq nG + \varepsilon B_X$.
This property was first considered by Schl\"{u}chtermann and Wheeler~\cite{sch-whe} and was studied further in \cite{avi-ple-rod-5,fab-mon-ziz,kam-mer2,mer-sta-2} among others. 
The basic examples of SWCG spaces are the reflexive spaces, the Schur separable spaces and $L_1(\mu)$ for any finite non-negative measure~$\mu$. All
SWCG spaces are weakly sequentially complete, \cite[Theorem~2.5]{sch-whe}. In general, the projective tensor 
product of two SWCG spaces is not SWCG, \cite[Example~2.11]{sch-whe}.
It is an open problem whether the Lebesgue-Bochner space $L_1([0,1],X)$ (which is isometrically isomorphic to $L_1[0,1] \pten X$) 
is SWCG whenever $X$ is SWCG, see \cite{rod17} and the references therein.
In the other direction, it is known that $\ell_2\pten \ell_2$ is SWCG, \cite[Example~2.3(c)]{sch-whe}. This result has been
extended to all preduals of $\sigma$-finite von Neumann algebras by de Pagter, Dodds and Sukochev, \cite[Proposition~4.3]{dep-dod-suk}, 
and to more general settings by Hamhalter, Kalenda, Peralta and Pfitzner, \cite[Theorem~9.3]{ham-alt}. 

It is natural to wonder whether $\ell_p\pten \ell_q$ is SWCG for arbitrary $1\leq p,q<\infty$. On the one hand, this is true if $p$ or~$q$ equals~$1$, because 
the $\ell_1$-sum of countably many SWCG spaces is SWCG, \cite[Proposition~2.9]{sch-whe}, and $\ell_1\pten X$ is isometrically isomorphic to~$\ell_1(X)$ 
for any Banach space~$X$. On the other hand, it is known that $\ell_p\pten \ell_q$ is 
reflexive if (and only if) $1/p+1/q<1$ (see, e.g., \cite[Corollary~4.24]{rya}). We will complete the picture by showing that $\ell_p\pten \ell_q$ is 
also SWCG in the remaining case (Corollary~\ref{cor:lplq}). This will be obtained as an immediate consequence of our main result
(see Section~\ref{section:preliminaries} for unexplained terminology):

\begin{theo}\label{theo:main}
Let $1 < p,q<\infty$ be such that $1/p+1/q\geq 1$. Let $X$ (resp.,~$Y$) be a Banach space with a countable unconditional finite-dimensional Schauder decomposition
having a disjoint lower $p$-estimate (resp.,  $q$-estimate). If $X$ and $Y$ are SWCG, then $X \pten Y$ is SWCG.
\end{theo}

As another application of the previous theorem, we get that $L_p[0,1] \pten L_q[0,1]$ is SWCG whenever $1< p,q \leq 2$ (Corollary~\ref{cor:LpLq}).
From the technical point of view, our approach to Theorem~\ref{theo:main} is inspired by some results about weak compactness in~$\ell_p\pten \ell_q$
due to Hamhalter and Kalenda~\cite{ham-kal}. This also allows to get results analogous to Theorem~\ref{theo:main} for the classes of 
strongly conditionally weakly compactly generated spaces, weakly sequentially complete spaces and Schur spaces  
(Theorems~\ref{theo:SCWCG}, \ref{theo:wsc} and~\ref{theo:Schur}).

The paper is organized as follows. In Section~\ref{section:preliminaries} we introduce some terminology and preliminary facts. Section~\ref{section:main} 
is devoted to proving Theorem~\ref{theo:main} and the aforementioned related results and applications.

\section{Terminology and preliminaries}\label{section:preliminaries} 

An {\em operator} is a linear continuous map between Banach spaces.
The topological dual of a Banach space~$X$ is denoted by~$X^*$. The norm of~$X$ is denoted by $\|\cdot\|_X$ or simply~$\|\cdot\|$. 
The closed unit ball of~$X$ is $B_X=\{x\in X:\|x\|\leq 1\}$. 
Given two sets $C_1,C_2 \sub X$, its Minkowski sum is $C_1+C_2:=\{x_1+x_2:\, x_1\in C_1, \, x_2\in C_2\}$. In particular, given a
set $C \sub X$ and $d\in \N$, we write 
$$
	\sum_{i=1}^d C:=\left\{\sum_{i=1}^d x_i: \, \text{$x_i\in C$ for all $i\in \{1,\dots,d\}$}\right\}.
$$
By a {\em subspace} of~$X$ we mean a norm closed linear subspace. By a {\em projection} from~$X$ onto a subspace~$Y \sub X$
we mean an operator $P:X\to X$ such that $P(X)=Y$ and $P$ is the identity when restricted to~$Y$.
The subspace of~$X$ generated by a set~$C\sub X$ is denoted by $\overline{{\rm span}}(C)$. 
A set $D \sub X$ is said to be {\em conditionally weakly compact} (or {\em weakly precompact})
if every sequence in~$D$ admits a weakly Cauchy subsequence; this is equivalent to saying that $D$ is bounded and contains no sequence
equivalent to the usual basis of~$\ell_1$ (by Rosenthal's theorem; see, e.g., \cite[Theorem~10.2.1]{alb-kal}).

\subsection{Unconditional Schauder decompositions}\label{subsection:Schauder}

Let $X$ be a Banach space. Recall that an {\em unconditional Schauder decomposition} of~$X$ is a family~$\{X_i\}_{i\in I}$ 
of subspaces of~$X$ such that each $x\in X$ can be written in a unique way as $x=\sum_{i\in I} x_i$, where $x_i\in X_i$ for all $i\in I$
and the series is unconditionally convergent. In this case, for each $i\in I$ one has a projection $P_i$ from~$X$ onto~$X_i$
in such a way that $x=\sum_{i\in I} P_i(x)$ for all $x\in X$. 
For each $C \sub I$, we denote by $P_C$ the projection from~$X$ onto $\overline{{\rm span}}(\bigcup_{i\in C}X_i)$
given by $P_C(x):=\sum_{i\in C}P_i(x)$ for all $x\in X$. One has $\sup\{\|P_C\|:C \sub I\}<\infty$ and so the formula $|||x|||=\sup\{\|P_C(x)\|: C\sub I\}$ 
defines an equivalent norm on~$X$. Therefore, we will always assume without loss of generality that $P_C$ has norm~$1$ for every non-empty set $C\sub I$. 
The decomposition is said to be {\em finite-dimensional} if each $X_i$ is finite-dimensional. The simplest examples of unconditional finite-dimensional Schauder decompositions are those induced by unconditional Schauder bases. Countable unconditional finite-dimensional Schauder decompositions
are usually called {\em unconditional FDDs}.

The following concept has been used in~\cite{joh-phi-sch}:

\begin{defi}\label{defi:p-lower}
Let $X$ be a Banach space and $1<p<\infty$. We say that an unconditional Schauder decomposition $\{X_i\}_{i\in I}$ of~$X$ 
has a {\em disjoint lower $p$-estimate} if there is a constant $c>0$ such that,
for every finite collection $C_1,\dots,C_r$ of pairwise disjoint finite subsets of~$I$ and for every $x\in X$, we have
$$
	\left(\sum_{k=1}^r \|P_{C_k}(x)\|^p\right)^{1/p} \leq c \|x\|.
$$ 
\end{defi}

Clearly, for any $1<p<\infty$ the usual basis of $\ell_p$ has a disjoint lower $p$-estimate. 

It is easy to check that
{\em if $X$ is a Banach space having cotype~$p < \infty$, then any unconditional Schauder decomposition of~$X$ has a disjoint lower $p$-estimate.}
For instance, if $\mu$ is any non-negative measure and $1\leq r<\infty$, then $L_r(\mu)$ has cotype $p=\max\{2,r\}$
(see, e.g., \cite[Theorem~6.2.14]{alb-kal}), hence any unconditional Schauder decomposition of~$L_r(\mu)$ has 
a disjoint lower $p$-estimate. 

\subsection{Projective tensor products}\label{subsection:pten}

Let $X$ and~$Y$ be two Banach spaces. We denote by $\mathcal{B}(X,Y)$
the Banach space of all continuous bilinear functionals on $X\times Y$, equipped with the norm
$\|S\|=\sup\{|S(x,y)|:\, x \in B_X,\, y \in B_Y\}$. 
Each element of~$\mathcal{B}(X,Y)$ induces
a linear functional (denoted in the same way) in the algebraic tensor product $X\otimes Y$.
The {\em projective tensor product} of~$X$ and~$Y$, denoted by $X\pten Y$, is the completion
of $X\otimes Y$ when equipped with the norm
$$
	\|z\|=\sup\{|S(z)|: \, S\in \mathcal{B}(X,Y), \, \|S\|\leq 1\}, \quad z \in X\otimes Y.
$$
Each $S \in \mathcal{B}(X,Y)$ defines an element of $(X\pten Y)^*$ and, in fact,
this gives an isometric isomorphism between $\mathcal{B}(X,Y)$ and~$(X \pten Y)^*$
(see, e.g., \cite[Section~2.2]{rya}).

Given two operators $T:X\to X_1$ and $S:Y \to Y_1$, where $X_1$ and $Y_1$ are Banach spaces, the {\em projective tensor product}
of $T$ and $S$ is the unique operator $T\otimes S: X\pten Y \to X_1 \pten Y_1$ satisfying $(T\otimes S)(x\otimes y)=T(x)\otimes S(y)$
for every $x\in X$ and for every $y\in Y$.

We will need the following well-known elementary fact: 

\begin{lem}\label{lem:tensor-compact} 
Let $X$ and $Y$ be Banach spaces. Let $W_1 \sub X$ be a relatively norm compact set, $W_2 \sub Y$ and consider
$W_1\otimes W_2=\{x\otimes y:x\in W_1, y\in W_2\} \sub X\pten Y$.  
\begin{enumerate}
\item[(i)] If $W_2$ is relatively weakly compact, then so is $W_1 \otimes W_2$.
\item[(ii)] If $W_2$ is weakly compact and $W_1$ is norm closed, then $W_1 \otimes W_2$ is weakly compact.
\item[(iii)] If $W_2$ is conditionally weakly compact, then so is $W_1\otimes W_2$.
\end{enumerate}
\end{lem}

\section{Main results}\label{section:main}

Throughout this section $X$ and $Y$ are Banach spaces with unconditional Schauder decompositions $\{X_i\}_{i\in I}$ and $\{Y_j\}_{j\in J}$
and associated projections $\{P_i\}_{i\in I}$ and $\{Q_j\}_{j\in J}$, respectively.
As we go through this section we will add further requirements to these decompositions (see Assumptions~\ref{ref:as1} and~\ref{ref:as2} below). 

\begin{rem}\label{rem:Rdefinition}
\rm For each $C \sub I$ and for each $D\sub J$, the operator
$$
	R_{C,D}:=P_C \otimes Q_D: X\pten Y \to X\pten Y
$$ 
is a projection (which has norm~$1$ if $C$ and $D$ are non-empty)
whose range is isometrically isomorphic in the natural way to $P_C(X) \pten Q_D(Y)$ (see, e.g., \cite[Proposition~2.4]{rya}). We will write
$$
	Z_{C,D}:=R_{C,D}\big(X\pten Y\big).
$$
\end{rem}

The identity operator on a Banach space~$Z$ will be denoted by~$I_Z$.

\begin{lem}\label{lem:decomposition}
Let $C \sub I$ and $D \sub J$. Then:
$$
	I_{X \pten Y}= I_X\otimes Q_D + P_C \otimes Q_{J\setminus D}+R_{I\setminus C,J\setminus D}.
$$
\end{lem}
\begin{proof}
Observe that 
$$
	I_{X\pten Y}=I_X\otimes Q_D + I_X \otimes Q_{J\setminus D}
	\quad\mbox{and}\quad
	I_X\otimes Q_{J\setminus D}=P_C \otimes Q_{J\setminus D}+R_{I\setminus C,J\setminus D}.
$$
\end{proof}

\begin{lem}\label{lem:approx}
Let $z\in X\pten Y$, $\epsilon>0$ and let $C \sub I$ and $D \sub J$ be finite. Then there exist
finite sets $C \sub C' \sub I$ and $D \sub D' \sub J$ such that $\|z-R_{C',D'}(z)\|_{X\pten Y} \leq \epsilon$.  
\end{lem}
\begin{proof}
Write $z=\sum_{n\in \N}x_n\otimes y_n$ for some sequences $(x_n)_{n\in \N}$ and $(y_n)_{n\in \N}$ in $X$ and~$Y$, respectively,
with $\sum_{n\in \N}\|x_n\|\|y_n\|<\infty$
(see, e.g., \cite[Proposition~2.8]{rya}).
Choose $N\in \N$ large enough such that 
\begin{equation}\label{eqn:star}
	\sum_{n>N}\|x_n\|\|y_n\|\leq \frac{\epsilon}{3}.
\end{equation}
Since $x=\sum_{i\in I}P_i(x)$ for all $x\in X$ and $y=\sum_{j\in J}Q_j(y)$ for all $y\in Y$, the series being unconditionally convergent,
we can take finite sets $C \sub C' \sub I$ and $D \sub D' \sub J$ in such a way that
$$
	\|x_n-P_{C'}(x_n)\| \|y_n\| \leq \frac{\epsilon}{6N}
	\quad \mbox{and}\quad
	\|y_n-P_{D'}(y_n)\| \|x_n\| \leq \frac{\epsilon}{6N}
$$ 
for every $n\in \{1,\dots,N\}$. Then
\begin{multline*}
	\big\|x_n\otimes y_n -P_{C'}(x_n)\otimes P_{D'}(y_n)\big\|_{X\pten Y} \\ \leq
	\big\|(x_n-P_{C'}(x_n))\otimes y_n\big\|_{X\pten Y} + \big\|P_{C'}(x_n)\otimes (y_n-P_{D'}(y_n))\big\|_{X\pten Y} \leq \frac{\epsilon}{3N}
\end{multline*}
for every $n\in \{1,\dots,N\}$, which together with~\eqref{eqn:star} allows us to conclude that
\begin{multline*}
	\|z-R_{C',D'}(z)\|_{X\pten Y}
	\leq \sum_{n=1}^N \big\|x_n\otimes y_n -P_{C'}(x_n)\otimes P_{D'}(y_n)\big\|_{X\pten Y} \\ +
	\left\|\sum_{n>N}x_n\otimes y_n\right\|_{X\pten Y} 
	+	\left\|R_{C',D'}\left(\sum_{n>N}x_n\otimes y_n\right)\right\|_{X\pten Y} \leq \epsilon,
\end{multline*}
as required.
\end{proof}

\begin{assumption}\label{ref:as1}
\rm Throughout the rest of this section we suppose that $\{X_i\}_{i\in I}$ 
has a disjoint lower $p$-estimate and that $\{Y_j\}_{j\in J}$ has a disjoint lower $q$-estimate, where $1< p,q<\infty$ satisfy $1/p+1/q\geq 1$.
\end{assumption}

\begin{lem}\label{lem:lower-estimates}
There is a constant $d>0$ such that, if
$(A_n)_{n\in\N}$ and $(B_n)_{n\in \N}$ are sequences of pairwise disjoint finite subsets of~$I$ and~$J$, respectively, then
$$
	\sum_{n\in\N}\left\|R_{A_n,B_n}(z)\right\|_{X\pten Y}\leq d\|z\|_{X\pten Y}
	\quad
	\text{for all $z\in X\pten Y$}.
$$ 
\end{lem}
\begin{proof}
Fix $z\in X\pten Y$ and $N\in \N$. We will check that
\begin{equation}\label{eqn:lower-estimates-N}
	\sum_{n=1}^N\left\|R_{A_n,B_n}(z)\right\|\leq c_X c_Y \|z\|_{X\pten Y},
\end{equation}
where $c_X$ and $c_Y$ are constants as in Definition~\ref{defi:p-lower} for $X$ and~$Y$, respectively.

For each $n\in \{1,\dots,N\}$ we take $S_n\in (Z_{A_n,B_n})^*$ with $\|S_n\|\leq 1$
such that 
$$
	S_n\big(R_{A_n,B_n}(z)\big)=\|R_{A_n,B_n}(z)\|_{X\pten Y}.
$$ 
Define $S\in (X\pten Y)^*$ by $S:=\sum_{n=1}^N S_n \circ R_{A_n,B_n}$.
Clearly, we have
$$
	\sum_{n=1}^N\left\|R_{A_n,B_n}(z)\right\|_{X\pten Y}=\sum_{n=1}^N S_n\big(R_{A_n,B_n}(z)\big)=
	S(z) \leq \|S\| \|z\|_{X\pten Y}.
$$
Therefore, in order to prove~\eqref{eqn:lower-estimates-N} it suffices to show that $\|S\|\leq c_X c_Y$. 
We identify $(X\pten Y)^*$ and $\mathcal{B}(X,Y)$ (see Subsection~\ref{subsection:pten}). Fix $x\in X$ and $y\in Y$.
Now, let $1<p^*<\infty$ be the conjugate of~$p$, that is, $1/p+1/p^*=1$. 
By H\"{o}lder's inequality, the fact that $q\leq p^*$ and the disjoint lower $p$-estimate (resp., $q$-estimate) of~$\{X_i\}_{i\in I}$ (resp.,~$\{Y_j\}_{j\in J}$), we have  
\begin{eqnarray*}
	|S(x,y)| & = & \left|\sum_{n=1}^N S_n\big(P_{A_n}(x) \otimes Q_{B_n}(y)\big)\right| \\ 
	& \leq & \sum_{n=1}^N \big|S_n\big(P_{A_n}(x) \otimes Q_{B_n}(y)\big)\big| \\
	& \leq & \sum_{n=1}^N \big\|P_{A_n}(x)\big\| \big\|Q_{B_n}(y)\big\| \\
	& \leq & \left(\sum_{n=1}^N \big\|P_{A_n}(x)\big\|^p\right)^{1/p} \left(\sum_{n=1}^N \big\|Q_{B_n}(y)\big\|^{p^*}\right)^{1/p^*} \\
	& \leq & \left(\sum_{n=1}^N \big\|P_{A_n}(x)\big\|^p\right)^{1/p} \left(\sum_{n=1}^N \big\|Q_{B_n}(y)\big\|^{q}\right)^{1/q} \\
	& \leq & c_X c_Y \|x\| \|y\|.
\end{eqnarray*}
This shows that $\|S\|\leq c_X c_Y$ and the proof is finished.
\end{proof}

The following result is based on some ideas of the proof of \cite[Theorem~2.1]{ham-kal} and will be a key tool for us.

\begin{theo}\label{theo:pten-rwc}
If $W \sub X\pten Y$ is conditionally weakly compact,
then for every $\epsilon >0$ there exist finite sets $C \sub I$ and $D \sub J$ such that
$$
	\|R_{I\setminus C,J\setminus D}(z)\|_{X\pten Y}\leq \epsilon
	\quad\text{for all $z\in W$.}
$$
\end{theo}
\begin{proof}
By contradiction, suppose that
$$
	c:=\inf\left\{\sup_{z\in W} \|R_{I\setminus C,J\setminus D}(z)\|_{X\pten Y}: \, \text{$C \sub I$ and $D \sub J$ are finite}\right\}>0.
$$
Fix $0<a<b<c$. With the help of Lemma~\ref{lem:approx}, we can construct inductively two sequences of finite sets
$$
	\emptyset=C_1 \sub C_2 \sub \dots \sub I,
	\quad\qquad
	\emptyset=D_1 \sub D_2 \sub \dots \sub J,
$$
and a sequence $(w_n)_{n\in \N}$ in~$W$ such that
\begin{equation}\label{eqn:big}
	\left\|R_{I \setminus C_n,J\setminus D_n}(w_n)\right\| > b
\end{equation}
and
\begin{equation}\label{eqn:small}
	\left\|w_n-R_{C_{n+1},D_{n+1}}(w_n)\right\| < a
\end{equation}
for every $n\in \N$. Define
$$
	A_n:=C_{n+1}\setminus C_n
	\quad\mbox{and}\quad 
	B_n:=D_{n+1}\setminus D_n
	\quad\text{for all $n\in \N$},
$$
so that $(A_n)_{n\in\N}$ and $(B_n)_{n\in \N}$ are sequences of pairwise disjoint finite subsets of~$I$ and~$J$, respectively.
For each $n\in \N$, the vector 
$$
	z_n:=R_{A_n,B_n}(w_n)
$$
satisfies
\begin{multline*}
	\left\|R_{I\setminus C_n,J\setminus D_n}(w_n)-z_n\right\|_{X\pten Y} \\ =  
	\left\|R_{I\setminus C_n,J\setminus D_n}(w_n)-R_{I\setminus C_n,J\setminus D_n}(R_{C_{n+1},D_{n+1}}(w_n))\right\|_{X\pten Y} \\ 
	 \leq \left\|R_{I\setminus C_n,J\setminus D_n}\right\|
	\left\|w_n-R_{C_{n+1},D_{n+1}}(w_n)\right\|_{X\pten Y}
	\stackrel{\eqref{eqn:small}}{<}a
\end{multline*}
and so
\begin{equation}\label{eqn:zn}
	\|z_n\|_{X\pten Y} \geq \left\|R_{I\setminus C_n,J\setminus D_n}(w_n)\right\|_{X\pten Y}-\left\|R_{I\setminus C_n,J\setminus D_n}(w_n)-z_n\right\|_{X\pten Y}
	\stackrel{\eqref{eqn:big}}{>} b-a.
\end{equation}

Let $U$ be the $\ell_1$-sum of the
sequence of Banach spaces $(Z_{A_n,B_n})_{n\in \N}$. Observe that Lemma~\ref{lem:lower-estimates} allows us to define an operator
$$
	T:X\pten Y \to U
$$
by the formula
$$
	T(z):=\left(R_{A_n,B_n}(z)\right)_{n\in \N}
	\quad
	\text{for all $z\in X\pten Y$}.
$$
Since $W$ is conditionally weakly compact, the same holds for $T(W)$
and then we can apply a folk characterization of conditionally weakly compact sets in $\ell_1$-sums of Banach spaces (see, e.g., \cite[Lemma~3.18]{rod20})
to find $N\in \N$ such that
$$
	\sup_{w\in W} \sum_{n>N} \left\|R_{A_n,B_n}(w)\right\|_{X\pten Y} \leq b-a.
$$
Therefore, we have
$$
	\left\|R_{A_n,B_n}(w)\right\|_{X\pten Y} \leq b-a
	\quad\text{for every $w\in W$ and for every $n>N$,}
$$
which contradicts~\eqref{eqn:zn}. The proof is finished.
\end{proof}

\begin{assumption}\label{ref:as2}
\rm Throughout the rest of the paper we assume further that 
all $X_i$'s and $Y_j$'s are finite-dimensional.
\end{assumption}

\subsection{Theorem~\ref{theo:main} and some applications}

It is convenient to introduce the following terminology:

\begin{defi}\label{defi:StronglyGenerated}  
Let $Z$ be a Banach space and let $\mathcal{C}$ and $\mathcal{G}$ be two families of subsets of~$Z$. We say that
{\em $\mathcal{C}$ is strongly generated by~$\mathcal{G}$} if for every $C \in \mathcal{C}$ and for every $\varepsilon>0$ there is 
$G \in \mathcal{G}$ such that $C \subseteq G + \varepsilon B_Z$.
\end{defi}

Thus, a Banach space~$Z$ is SWCG if and only if the family of all weakly compact subsets of~$Z$ is strongly generated by $\{nG:n\in \N\}$ for some
weakly compact set $G \sub Z$ (in this case, we say that $Z$ is SWCG by~$G$). It is known that $Z$ is SWCG 
if and only if the family of all weakly compact subsets of~$Z$ is strongly generated by a {\em countable} subfamily, \cite[Theorem~2.1]{sch-whe}.

As we already mentioned in the introduction, the projective tensor product of a finite-dimensional Banach space
and a SWCG Banach space is SWCG. The next lemma describes ``strongly generating'' sets for such spaces.

\begin{lem}\label{lem:strong-generation-fd}
Let $U$ be a finite-dimensional Banach space with $d={\rm dim}(U)$ and let $V$ be a Banach space which is SWCG by the 
weakly compact set~$G \sub V$. 
Then $U\pten V$ is SWCG by the weakly compact set $G':=\sum_{i=1}^d B_U\otimes G$.
\end{lem}
\begin{proof}
Observe that $G'$ is weakly compact in $U\pten V$, since it is the Minkowski sum of finitely many
copies of the weakly compact set $B_U\otimes G$ (apply Lemma~\ref{lem:tensor-compact}).
By renorming, we can assume that $U=\ell_1^d$. There is an isometric isomorphism 
$$
	J:U\pten V \to W:=\underbrace{V \oplus_1 \dots \oplus_1 V}_{\text{$d$ times}}
$$
satisfying 
$$
	J\big((a_k)_{k=1}^d \otimes v\big)=(a_1v,\dots,a_dv)
	\quad
\text{for every $(a_k)_{k=1}^d \in U$ and for every $v\in V$}
$$
(see, e.g., \cite[Example~2.6]{rya}). We will show that the 
family of all weakly compact subsets of~$W$ is strongly generated by $\{nJ(G'):n\in \N\}$.
Indeed,
let $K \sub W$ be a weakly compact set and fix $\epsilon>0$. For each $k\in \{1,\dots,d\}$,
let $\pi_k:W \to V$ be the $k$th-coordinate projection, so that $\pi_k(K)$ is a weakly compact subset of~$V$ and
we can choose $m_k\in \N$ such that 
\begin{equation}\label{eqn:pij}
	\pi_k(K) \sub m_k G +\frac{\epsilon}{d}B_V.
\end{equation}
Pick $m\in \N$ with $m \geq m_k$ for every $k\in \{1,\dots,d\}$.
Let $\{e_1,\dots,e_d\}$ be the usual basis of~$U=\ell_1^d$.
Then 
\begin{eqnarray*}
	K  &\stackrel{\eqref{eqn:pij}}{\sub}& \big\{(m_1v_1, \dots,m_d v_d): \, v_k \in G \text{ for all $k\in \{1,\dots,d\}$}\big\} + \epsilon B_W \\
	 &= &\left\{\sum_{k=1}^d J(m_k e_k \otimes v_i) : \, v_k \in G \text{ for all $k\in \{1,\dots,d\}$}\right\} + \epsilon B_W \\
	&\sub& m \sum_{k=1}^d J(B_U \otimes G) + \epsilon B_W=mJ(G')+\epsilon B_W.
\end{eqnarray*}
The proof is complete.
\end{proof}

We are now ready to prove our main result:

\begin{proof}[Proof of Theorem~\ref{theo:main}]
By assumption, both $I$ and $J$ are countable, all $X_i$'s and $Y_j$'s are finite-dimensional and
$X$ (resp.,~$Y$) is SWCG by some weakly compact 
set $G_X \sub X$ (resp., $G_Y \sub Y$). For any finite sets $C \sub I$ and $D \sub J$ and any $n,m\in \N$, the set
$$
	H_{C,D,n,m}:= n\left(\sum_{k=1}^{{\rm dim}(P_C(X))} B_{P_C(X)} \otimes G_Y\right) + m\left(\sum_{k=1}^{{\rm dim}(Q_D(Y))} G_X\otimes B_{Q_D(Y)}\right)
$$ 
is weakly compact in~$X\pten Y$ (apply Lemma~\ref{lem:tensor-compact} and the fact that the Minkowski sum of finitely
many weakly compact sets is weakly compact). 

In order to prove that $X\pten Y$ is SWCG we will check that the (countable) family of all sets of the form $H_{C,D,n,m}$ 
strongly generates the family of all weakly compact subsets of~$X\pten Y$.
Let $W \sub X\pten Y$ be a weakly compact set and fix $\epsilon>0$. 
By Theorem~\ref{theo:pten-rwc}, there exist finite sets $C \sub I$ and $D \sub J$ such that
\begin{equation}\label{eqn:remainder}
	R_{I\setminus C,J\setminus D}(W) \sub \frac{\epsilon}{3} B_{X\pten Y}. 
\end{equation}
Then $W':=(I_X\otimes Q_{J\setminus D})(W)$
is weakly compact and, by Lemma~\ref{lem:decomposition}, we have
\begin{equation}\label{eqn:biginclu}
	W \sub (I_X\otimes Q_D)(W) + (P_C\otimes I_Y)(W')+R_{I\setminus C,J\setminus D}(W).
\end{equation}
Since $(I_X\otimes Q_D)(W)$ and $(P_C\otimes I_Y)(W')$ are weakly compact subsets of $Z_{I,D}$ and $Z_{C,J}$, respectively, 
Lemma~\ref{lem:strong-generation-fd} ensures the existence of $n,m\in \N$ such that
$$
	(I_X\otimes Q_D)(W) \sub m \left(\sum_{k=1}^{{\rm dim}(Q_D(Y))} G_X\otimes B_{Q_D(Y)}\right)+\frac{\epsilon}{3}B_{X\pten Y}
$$
and 
$$
	(P_C\otimes I_Y)(W') \sub n \left(\sum_{k=1}^{{\rm dim}(P_C(X))} B_{P_C(X)} \otimes G_Y\right)+\frac{\epsilon}{3}B_{X\pten Y}.
$$
These inclusions, \eqref{eqn:remainder} and~\eqref{eqn:biginclu} imply that 
$W \sub H_{C,D,n,m}+\epsilon B_{X\pten Y}$.
\end{proof}

By putting together Theorem~\ref{theo:main} and the information already collected in the introduction and Subsection~\ref{subsection:Schauder},
we get the following corollaries. 

\begin{cor}\label{cor:lplq}
The space $\ell_{p_1}\pten \ell_{p_2}$ is SWCG for any $1\leq p_1,p_2<\infty$.
\end{cor}

Bearing in mind that the space $L_r[0,1]$ has an unconditional basis if $1<r<\infty$ (see, e.g., \cite[Theorem~6.1.6]{alb-kal}), we also have:

\begin{cor}\label{cor:LpLq}
The space $L_{p_1}[0,1] \pten L_{p_2}[0,1]$ is SWCG for any $1<p_1, p_2 \leq 2$.
\end{cor}

The previous statement is also true for $p_1=1$ and any $1\leq p_2 <\infty$, because
the space $L_1([0,1],Z)$ is SWCG whenever $Z$ is a reflexive Banach space or $Z=L_1(\mu)$ for a finite non-negative measure~$\mu$,
see \cite[Section~3]{sch-whe}.

\begin{cor}\label{cor:Lplq}
The space $L_{p_1}[0,1]\pten \ell_{p_2}$ is SWCG for any $1< p_1,p_2<\infty$ satisfying $1/\max\{2,p_1\}+1/p_2\geq 1$.
\end{cor}

As a consequence of the results of~\cite{sch-whe}, the previous statement is also true for $p_1=1$ and any $1\leq p_2 <\infty$ or vice versa.

\subsection{Further results}

Following~\cite{kun-sch}, a Banach space $Z$ is said to be {\em strongly conditionally weakly compactly generated} ({\em SCWCG} for short) 
if there is a conditionally weakly compact set $G \sub Z$ such that, 
for every conditionally weakly compact set $C \sub X$ and for every $\varepsilon>0$, there is $n\in \mathbb{N}$ such that $C \subseteq nG + \varepsilon B_Z$.
It turns out that a Banach space is SWCG if and only if it is SCWCG and weakly sequentially complete, \cite[Theorem~2.2]{laj-rod-2}.
The proof of Theorem~\ref{theo:main} can be adapted straightforwardly to get the following result:

\begin{theo}\label{theo:SCWCG}
Suppose that $I$ and $J$ are countable. If $X$ and $Y$ are SCWCG, then
$X\pten Y$ is SCWCG.
\end{theo}

Pisier, \cite[Section~4]{pis4}, showed that, in general, the projective tensor product of two weakly sequentially complete Banach spaces is not weakly sequentially complete
(that example is the one used in~\cite[Example~2.11]{sch-whe} for the property of being SWCG).
As to positive results, it is known that the projective tensor product of two weakly sequentially complete Banach spaces is
weakly sequentially complete provided that one of them has an unconditional basis, thanks to a result of Lewis, \cite[Corollary~10]{lew5}. Talagrand proved that the space
$L_1([0,1],Z)$ is weakly sequentially complete whenever $Z$ is a weakly sequentially complete Banach space, \cite[Theorem~11]{tal11}.
We can give another positive result along this way:

\begin{theo}\label{theo:wsc}
If $X$ and $Y$ are weakly sequentially complete, then so is $X\pten Y$.
\end{theo}
\begin{proof} We will show that every conditionally weakly compact set $W \sub X\pten Y$
is relatively weakly compact. Fix $\epsilon>0$. By Theorem~\ref{theo:pten-rwc} there exist
finite sets $C \sub I$ and $D \sub J$ such that $R_{I\setminus C,J\setminus D}(W) \sub \epsilon B_{X\pten Y}$. 
By Lemma~\ref{lem:decomposition}, we have
$$
	W \sub (I_X\otimes Q_D)(W) + (P_C\otimes Q_{J\setminus D})(W) + \epsilon B_{X\pten Y}.
$$
Since $Z_{I,D}$ and $Z_{C,J}$ are weakly sequentially complete (because
the projective tensor product of a weakly sequentially complete space and a finite-dimensional Banach space is weakly sequentially complete), 
both $(I_X\otimes Q_D)(W)$ and $(P_C\otimes Q_{J\setminus D})(W)$ are relatively weakly compact. Therefore,
$H:=(I_X\otimes Q_D)(W) + (P_C\otimes Q_{J\setminus D})(W)$ is a relatively weakly compact subset of~$X\pten Y$ satisfying
$W \sub H+\epsilon B_{X\pten Y}$. As $\epsilon>0$ is arbitrary, we conclude that $W$ is relatively weakly compact in $X\pten Y$
(see, e.g., \cite[Lemma~13.32]{fab-ultimo}).
\end{proof}

\begin{exa}\label{exa:lplq-uncountable}
\rm Let $\Gamma$ be an uncountable set and let $1<p_1,p_2<\infty$ be such that $1/p_1+1/p_2\geq 1$. 
Then $Z:=\ell_{p_1}(\Gamma)\pten \ell_{p_2}(\Gamma)$ is not SCWCG. Indeed, this space is not weakly compactly generated
because it contains a subspace isomorphic to~$\ell_1(\Gamma)$ (see, e.g., \cite[Proposition~3.6]{avi-mar-rod-rue}). Therefore,
$Z$ is not SWCG. However, it is weakly sequentially complete, because
every sequence in~$Z$ is contained in a subspace isomorphic to $\ell_{p_1}\pten \ell_{p_2}$, which is weakly sequentially complete
by Lewis' aforementioned result. From \cite[Theorem~2.2]{laj-rod-2} we conclude that $Z$ cannot be SCWCG.
\end{exa}

It seems to be an open problem whether the Schur property is preserved by projective tensor products, \cite{bot-rue,gon-gut}. 
The argument of Theorem~\ref{theo:wsc} can be imitated to prove that every conditionally weakly compact subset of~$X\pten Y$ is relatively norm compact
provided that $X$ and $Y$ have the Schur property, so one gets:
 
\begin{theo}\label{theo:Schur}
If $X$ and $Y$ have the Schur property, then so does $X\pten Y$.
\end{theo}

We finish the paper with another application of Theorem~\ref{theo:pten-rwc}.

\begin{pro}\label{pro:tensor-rwc}
Let $W_1 \sub X$ and $W_2 \sub Y$ be sets such that
$W_1\otimes W_2$ is conditionally weakly compact in $X\pten Y$. Then either $W_1$ or $W_2$ is relatively norm compact.
\end{pro}
\begin{proof}
By contradiction, suppose that both $W_1$ and $W_2$ fail to be relatively norm compact. Then there exist $\eta>0$ such that $W_1 \not\subseteq K_1 + \eta B_X$
and $W_2 \not\subseteq K_2 + \eta B_Y$ for every relatively norm compact sets $K_1 \sub X$ and $K_2 \sub Y$. In particular, if $C \sub I$ and $D \sub J$
are finite, then $W_1 \not\subseteq P_C(W_1) + \eta B_X$
and $W_2 \not\subseteq P_D(W_2) + \eta B_Y$, hence there exist $w_1\in W_1$ and $w_2\in W_2$ such that
$$
	\|w_1-P_C(w_1)\|_{X\pten Y}> \eta \quad\mbox{and}\quad
	\|w_2-P_D(w_2)\|_{X\pten Y}> \eta,
$$
and so
\begin{eqnarray*}
	\left\|R_{I \setminus C,J\setminus D}(w_1\otimes w_2)\right\|_{X\pten Y}  & = &
	\left\|P_{I\setminus C}(w_1)\otimes Q_{J\setminus D}(w_2)\right\|_{X\pten Y} \\ & = &
	\|w_1-P_C(w_1)\|_{X\pten Y} \|w_2-Q_D(w_2)\|_{X\pten Y} \geq \eta^2.
\end{eqnarray*}
This contradicts the conclusion of Theorem~\ref{theo:pten-rwc} applied to $W_1\otimes W_2$.
\end{proof}

The embeddability of~$\ell_1$ into projective tensor products was studied in \cite{emm2,bu-alt}. In this direction, we have:

\begin{cor}\label{cor:l1}
If $U \sub X$ and $V \sub Y$ are infinite-dimensional subspaces,
then $U\pten V$ contains a subspace isomorphic to~$\ell_1$.
\end{cor}
\begin{proof}
Let $i_U: U \to X$ and $i_V:V \to Y$ be the inclusion operators and consider the operator $T:=i_U\otimes i_V:U \pten V \to X\pten Y$.
Since $B_U \otimes B_V \sub T(B_{U\pten V})$ and both $B_U$ and $B_V$ fail to be norm compact,
from Proposition~\ref{pro:tensor-rwc} it follows that $T(B_{U\pten V})$ is not conditionally weakly compact. Hence
$B_{U\pten V}$ is not conditionally weakly compact and so $U\pten V$ contains a subspace isomorphic to~$\ell_1$.
\end{proof}

\subsection*{Funding} The research is partially supported by grants MTM2017-86182-P 
(funded by MCIN/AEI/10.13039/501100011033 and ``ERDF A way of making Europe'') and 
20797/PI/18 (funded by {\em Fundaci\'on S\'eneca}).

\subsection*{Acknowledgements}
I thank A.~Avil\'{e}s, G.~Mart\'{i}nez-Cervantes and A.~Rueda Zoca
for valuable discussions related to the topic of this note.

\bibliographystyle{amsplain}

\end{document}